\documentclass[preprint,12pt]{elsarticle}

\usepackage{color} 
\RequirePackage{amsthm,amsmath, amssymb, enumerate, natbib}
\RequirePackage[colorlinks,citecolor=blue,urlcolor=blue]{hyperref}
\usepackage{float}
\usepackage{times}

\usepackage{graphicx} 

\newtheorem{proposition}{Proposition}[section]

\newcommand{\n}{^{(n)}}

\newcommand{\pr}{^{\prime}}
\newcommand{\ppr}{^{\prime\prime}}

\newcommand{\ut}[1] {\renewcommand{\arraystretch}{0.5}
\begin{array}[t]{c}{#1}\\
   \widetilde{}
\end{array}
\renewcommand{\arraystretch}{1}}

\begin{document}

\begin{frontmatter}



\title{\sc On  Hodges and Lehmann's  ``6/$\pi$ Result" }


\author{Marc Hallin$^{a}\!$, \fnref{a1}
Yvik Swan$^{b}\!$,\fnref{a2} 
 and Thomas Verdebout\fnref{a3}}

\address[a1]{ECARES, Universit\' e libre de Bruxelles 
 and ORFE, Princeton University}

\address[a2]{Universit\'e du Luxembourg}

\address[a3]{Universit\'e Lille Nord de Fance, laboratoire EQUIPPE
}

\begin{abstract}
 While the asymptotic relative efficiency (ARE) of Wilcoxon rank-based
 tests for location and regression with respect to their parametric
 Student competitors can be arbitrarily large, Hodges and 
Lehmann (1961) have shown that the ARE of the same Wilcoxon tests with
respect to their van der Waerden or normal-score counterparts is
bounded from above by $6/\pi\approx 1.910$. In this paper, we revisit
that result, and investigate similar bounds for statistics based on
Student scores. We also consider the serial version of this ARE. More
precisely, we study the ARE, under various densities,  of the
Spearman-Wald-Wolfowitz and Kendall rank-based autocorrelations  with
respect to the van der Waerden or normal-score ones  used to test
(ARMA) serial dependence alternatives.  

\end{abstract}


\end{frontmatter}
 
\section{Introduction}\label{intro}
  
The Pitman asymptotic relative efficiency $\text{ARE}_f(\phi_{1} /
\phi_{2})$ under density $f$ of a test  $\phi_{1} $ with respect to a test $\phi_{2} $ is defined as the limit  (when it exists)   as $n_1$ tends to infinity  of the ratio 
 $n_{2;f}( n_1) / n_1$ of the number $n_{2;f}( n_1)$ of observations  it
 takes for the test $\phi_{2} $, under density $f$,  to match the
 local performance of the test $\phi_{1} $ based on~$n_1$
 observations. That concept  was first proposed by Pitman  in  the
 unpublished lecture notes~\cite{Pit49} he prepared for a~1948-49
 course at Columbia University. The first published rigorous treatment
 of the subject was by Noether~\cite{Noe} in 1955.  A similar
 definition applies to point estimation; see, for instance,
 \cite{H012} for a more precise
 definition.  An in-depth treatment 
 of the concept can be found in Chapter~10 of Serfling \cite{Serf},
 Chapter~14 of van der Vaart \cite{vdV}, or in the monograph by
 Nikitin~\cite{Niki}.

The study of the AREs of rank tests and R-estimators with respect to each other or with respect to their classical Gaussian counterparts   has produced  a number of interesting and sometimes   surprising results.   Considering  the van der Waerden or normal-score two-sample location rank test $\phi_{\rm vdW}$ and its classical normal-theory competitor, the two-sample Student test $\phi_{\mathcal{N}}$, 
Chernoff and Savage in 1958  established the rather  striking fact that, under any density~$f$ satisfying very mild regularity
assumptions, 
\begin{equation}\label{eq:1.3}
  \text{ARE}_f(\phi_{\rm vdW}/ \phi_{\mathcal{N}}) \ge 1,
\end{equation}
with equality holding  at the Gaussian density $f = \phi$ only. That result implies
that   rank  tests based on Gaussian scores (that is, the  two-sample   rank-based tests  for location, but also the    one-sample   signed-rank ones,   traditionally associated with the names of van der Waerden, Fraser, Fisher, Yates, Terry and/or Hoeffding---for simplicity, in the sequel, we uniformly  call them {\it van der Waerden tests}---asymptotically outperform the corresponding  everyday practice Student~$t$-test; see~\cite{CS58}. That result readily extends to one-sample symmetric and $m$-sample location, regression and analysis of variance models with independent noise.  

Another celebrated bound is the one obtained in 1956 by Hodges and Lehmann, who proved 
that, denoting by $\phi_{\rm W}$ the Wilcoxon test (same location and regression problems as above), 
\begin{equation}\label{eq:1.4}
  \text{ARE}_f(\phi_{\rm W}/ \phi_{\mathcal{N}}) \ge  0.864,
\end{equation}
which  implies that the   price to be paid for using   rank- or signed-rank tests of the Wilcoxon type (that is, logistic-score-based rank tests) instead of the traditional Student ones never exceeds 13.6\% of the total number of observations.  That bound moreover is sharp, being reached under the Epanechnikov density $f$. On the other hand, the benefits of considering Wilcoxon rather than Student can be arbitrarily large, as it is easily shown that the supremum over $f$ of $  \text{ARE}_f(\phi_{\rm W}/ \phi_{\mathcal{N}}) $ is infinite; see \cite{HL56}.

Both  
\eqref{eq:1.3} and \eqref{eq:1.4}    created quite  a surprise in the statistical community of  the late fifties, and helped  dispelling the wrong idea,  by then quite widespread,  that rank-based methods, although convenient and robust,  could not be expected to compete with the efficiency of  traditional parametric procedures. 

Chernoff-Savage and Hodges-Lehmann inequalities since then have been
extended  to a variety of more general settings. In the elliptical
context,  optimal rank-based procedures for location (one and
$m$-sample case), regression,   
 and scatter (one and
$m$-sample cases)  have been constructed in a series of papers by
Hallin and Paindaveine (\cite{HP02},  
\cite{HP06}, and \cite{HP08}), based on a multivariate concept of
signed ranks. The   Gaussian competitors here  are of the Hotelling,
Fisher,  or Lagrange
multiplier forms. For all those tests, Chernoff-Savage result similar
to \eqref{eq:1.3} have been established (see also
\cite{Pa04,Pa06}). Hodges-Lehmann results also have been obtained,
with bounds that, quite  interestingly, depend on the dimension of the
observation space: see~\cite{HP02}.

Another   type of extension is into the direction of time series and linear rank  statistics of the serial type.   Hallin \cite{H94} extended
Chernoff and Savage's result \eqref{eq:1.3}  to the serial context by
showing that the  serial van der Waerden  rank tests  
also uniformly dominate their  Gaussian competitors
(of the correlogram-based  portmanteau, Durbin-Watson or Lagrange
multiplier forms). Similarly,  
 Hallin and Tribel \cite{HaTr00}  proved  that the 0.864 upper bound in \eqref{eq:1.4}
no longer holds for the AREs of the   Wilcoxon  serial rank test 
with respect to their Gaussian competitors, and is to be replaced by  a slightly lower  0.854 one. Elliptical versions of those results are derived in  Hallin and Paindaveine (\cite{HP02'}, \cite{HP04}, \cite{HP05}).

Now, AREs  with respect to Gaussian procedures such as  $t$-tests are not always the best  evaluations of  the asymptotic performances of  rank-based tests. Their existence indeed requires  the Gaussian procedures  to be valid under the density~$f$ under consideration, a condition which places restrictions on $f$ that may not be satisfied. When the 
 Gaussian tests are no longer valid,  
 one rather may like  to consider   AREs~of the form 
\begin{equation}
  \label{eq:16}
  \text{ARE}_f(\phi_J / \phi_{K}) =1/ \text{ARE}_f(\phi_K / \phi_{J}) 
\end{equation}
 comparing the asymptotic performances (under $f$) of two rank-based tests $\phi_J $ and~$\phi_{K}$, based on score-generating  functions $J$ and $K$, respectively. 
Being distri\-bution-free,   rank-based procedures indeed do not impose any validity conditions on $f$, so that $ \text{ARE}_f(\phi_J / \phi_{K})$ 
in general  exists under much milder requirements on~$f$;  see, for instance,  \cite{HSVV11a} and~\cite{HSVV11b}, where AREs of the form (\ref{eq:16}) are provided for  rank-based methods  in linear
models with stable errors  under which  Student tests are not valid.

Obtaining bounds for  
  ARE$_f(\phi_J / \phi_{K})$,  in general, is not as easy as for AREs of the form
   ARE$_f(\phi_{J}/ \phi_{\mathcal{N}})$. The first result of that type  
 was established in 1961 by  Hodges and   Lehmann, who in \cite{HL61} show  that 
\begin{equation}
  \label{eqbis}
 0 \le   \text{ARE}_f(\phi_{\rm W} /\phi_{\rm vdW}) \leq  {6}/{\pi}\approx 1.910
\end{equation}
 or, equivalently,
\begin{equation}
  \label{eqter}
0.524 \approx \pi/6 \le   \text{ARE}_f(\phi_{\rm vdW} /\phi_{\rm W}) \leq  \infty 
\end{equation}
for all $f $ in some class $ \mathcal{F}$ of density functions
satisfying weak differentiability conditions.   Hodges and Lehmann moreover exhibit a parametric
 family of densities~${\cal F}_{\rm HL}=\{f_\alpha \vert\, \alpha\in [0,\infty)\}$ 
   for which the function~$\alpha \mapsto \text{ARE}_{f_{\alpha}}(\phi_{\rm W} /\phi_{\rm vdW})$ 
   achieves any value in the open interval $(0, {6}/{\pi})$ ($\alpha \mapsto \text{ARE}_{f_{\alpha}}(\phi_{\rm vdW} /\phi_{\rm W})$ 
   achieves any value in the open interval $(\pi/6, \infty)$). The  
 lower and upper    bounds  in (\ref{eqbis}) and~(\ref{eqter}) thus are  \textit{sharp} in the sense that they are the best possible ones.   The same result was extended and generalized by Gastwirth~\cite{G70}.

  Note that, in case $f$ has finite second-order moments (so that $ \text{ARE}_f(\phi_{\rm W} /\phi_{\cal N})$ is well defined),    since  
 $ \text{ARE}_f(\phi_{\rm vdW}/\phi_{\cal N}) = \text{ARE}_f(\phi_{\rm vdW}/\phi_{\rm W} )\times \text{ARE}_f(\phi_{\rm W} /\phi_{\cal N})
$, 
Hodges and   Lehmann's ``$6/\pi$ result"   implies that the ARE of the van der Waerden tests with respect to the Student ones, which by the Chernoff-Savage inequality is larger than or equal to one, actually  can be   arbitrarily large, and that this happens for the same types of densities as for the Wilcoxon tests. This is an indication that, when Wilcoxon is quite significantly outperforming Student, that performance is shared by a broad class of rank-based tests and $R$-estimators, which includes the van der Waerden ones.

In Section~\ref{sec:asympt-relat-effic}, we successively consider   the traditional case of  \textit{nonserial} rank statistics used in the context of location and regression models with independent observations,  and the case of  {\it serial} rank  statistics; the latter  involve ranks at time~$t$ and~$t-k$, say, and aim at detecting serial dependence among the observations. Serial rank statistics typically involve two score functions and, instead of  (\ref{eq:16}), yield AREs of the form
\begin{equation}\label{eq:16ser}
 \text{ARE}^*_f(\phi_{J_1,J_2}/\phi_{J_3, J_4}).
\end{equation}

 To start with, in Section~\ref{location},  we   revisit Gastwirth's classical    nonserial results. More precisely, we provide (Proposition~\ref{bounds}) a slightly different proof of the main proposition  in~\cite{G70}, with some further illustrations in the case of  Student scores. In Section~\ref{sercase}, we turn  to the serial case,  with special attention for the so-called Wilcoxon-Wald-Wolfowitz, Kendall and van der Waerden rank autocorrelation  coefficients. Serial AREs of the form~(\ref{eq:16ser}) typically are the product of two factors to which  the nonserial techniques of Section~\ref{location} separately apply; this provides  bounds which, however, are not sharp. Therefore, in Section~\ref{sec:numer-cons}, we restrict to a few parametric families---the Student family (indexed by the degrees of freedom), the power-exponential family, or the Hodges-Lehmann family ${\cal F}_{\rm HL}$---for which numerical values are displayed.

\section{Asymptotic relative efficiencies of  rank-based procedures}
\label{sec:asympt-relat-effic}

\setcounter{equation}{0}

The asymptotic behavior of rank-based test
statistics under local alternatives, since H\' ajek and \v Sid\' ak
\cite{HvS}, is   obtained via an application of Le~Cam's
Third Lemma (see, for instance,  Chapter 13 of~\cite{vdV}).  Whether the statistic is of the serial or the nonserial type, the result, under a density $f$ with distribution  function~$F$  involves integrals of the form 
$$\mathcal{K}(J) :=
\int_{0}^1J^2(u){\rm d}u\  ,\qquad \mathcal{K}(J, f)
:=\int_{0}^1 J(u)\varphi_{f}(F^{-1}(u)){\rm d}u 
,$$ and, in the serial case, $$
\mathcal{J}(J, f)
:=\int_{0}^1 J(u)F^{-1}(u){\rm d}u
$$
where, assuming that $f$ admits a
weak derivative~$f\pr$,  $\varphi_f:=  -  f\pr /f$ is such that  the 
Fisher information for location ${\cal 
  I}(f):=\int_0^1\varphi_f^2(F^{-1}(u)) {\rm d}u$ is finite. Denote by~${\cal F}$ the class of such densities. If local alternatives, in the serial case,  are of the ARMA type, $f$ is further restricted to   the subset ${\cal F}_2$ of densities $f\in{\cal F}$ having finite second-order moments.   Differentiability in quadratic mean of
$f^{1/2}$ is the standard assumption here, see Chapter 7 of
  \cite{vdV}; but absolute continuity of $f$ in the traditional
sense, with    a.e. derivative $f\pr$, is sufficient  for most purposes. We refer to  \cite{HvS} and   \cite{HP94}  for details in the nonserial and the serial case, respectively.

\subsection{The nonserial case} \label{location}
In location or regression problems, or, more generally, when testing linear constraints on the parameters of a linear model  (this includes ANOVA etc.),  the 
 ARE, under  density $f\in{\cal F}$,  of a
rank-based test $\phi_{J_1}$ based on the square-summable score-generating function $J_1$
with respect to another rank-based test $\phi_{{J_2}}$ based on
the square-summable  score-generating function ${J_2}$  takes the form 
\begin{equation}\label{AREf}  {\rm
    \text{ARE}}_{f}\left(\phi_{J_1}/\phi_{J_2}\right) =    \frac{\mathcal{K}({J}_2)}{\mathcal{K}({J}_1)}C^2_f(J_1, J_2),\quad\text{with}\quad C_f(J_1, J_2):=
  \frac{\mathcal{K}(J_1, f)}{ \mathcal{K}(J_2, f)}, \end{equation} 
provided that   $J_1$ and $J_2$ are monotone, or the difference between two monotone functions. Those ARE values readily extend to the $m$-sample setting, and to  R-estimation  problems. In a time-series context with innovation density $f\in{\cal F}_2$,  and under slightly more restrictive assumptions on the scores, they also extend to the partly rank-based tests and   R-estimators considered by Koul and Saleh in  \cite{KS93} and \cite{KS95}. 

 Gastwirth (1970) is basing his analysis of (\ref{AREf}) on an integration by parts of the integral in the definition of ${ \mathcal{K}(J, f)}$.  If both $J_1$ and $J_2$ are differentiable, with derivatives $J_1\pr$ and $J_2\pr$, respectively, and provided that $f$ is such that 
$$\lim_{x\rightarrow\infty}J_1(F(x))f(x) =0= \lim_{x\rightarrow\infty}J_2(F(x))f(x),
$$
integration by parts in those   integrals 
 yields, for  \eqref{AREf}, 
\begin{align}  {\rm
    \text{ARE}}_{f}\left(\phi_{J_1}/\phi_{J_2}\right) 
=  \frac{\mathcal{K}({J}_2)}{\mathcal{K}({J}_1)} 
  \left(\frac{ \int_{-\infty}^\infty  J_1\pr(F(x)) f^2(x){\rm d}x}{ \int_{-\infty}^\infty
      J_2\pr(F(x)) f^2(x){\rm d}x }\right)^{2}.\label{AREf3}  \end{align}

   In view of the Chernoff-Savage result \eqref{eq:1.3}, the   van~der~Waerden  
  score-genera\-ting function
  \begin{equation}
    \label{eq:3}
    J_2(u) = J_{\rm vdW}(u) = \Phi^{-1}(u)
  \end{equation}
  (with $u \mapsto \Phi^{-1}(u)$ the standard normal  quantile function)
may appear as a natural benchmark for ARE computations. From a technical point of view, under this integration by parts approach, the Wilcoxon score-generating function
  \begin{equation}\label{scWil}  
  J_2(u) = J_{\rm W}(u) = u-1/2
  \end{equation}
 (the Spearman-Wald-Wolfowitz  score-generating function in the serial case) is more appropriate, though. Convexity arguments indeed will play an important role, and, being linear, $J_{\rm W}$ is both convex and concave.   
Since  $J_{\mathrm{W}}\pr(u) = 1$ and~$\mathcal{K}(J_{\mathrm{W}})=1/12$,  
    equation  \eqref{AREf3} yields
\begin{align} {\rm 12\, \text{ARE}}_{f}\left(\phi_{J_1}/\phi_{{\rm
        W}}\right) 
  = \frac{1}{\mathcal{K}({J}_1)} \left(\frac{ \int_{-\infty}^\infty
      J_1\pr(F(x)) f^2(x){\rm d}x}{ \int_{-\infty}^\infty f^2(x){\rm
        d}x }\right)^{2}. 
        \label{AREff3} \end{align}
   Bounds  on $J_1\pr(F(x))$ then readily yield bounds on AREs, irrespective of $f$. 
 
 That property of Wilcoxon scores is exploited in   Propositions \ref{bounds} and \ref{bounds'} for non-serial AREs, in Propositions \ref {serbounds}  for  the serial   ones; those bounds are mainly about AREs of, or with respect to, Wilcoxon (Spearman-Wald-Wolfowitz) procedures, but not exclusively so.

Assume that $f\in{\cal F}_0:=\{f\in{\cal F}\vert \ \lim_{x\to\pm\infty}f(x) = 0\}$. Then, integration by parts is possible in the definition of ${\cal K}(J_{\rm W},f)$, yielding 
$${\cal K}(J_{\rm W},f)=\int_{-\infty}^\infty
  f^2(x) {\rm d}x.$$
  Assume furthermore that the square-integrable score-generating function $J_1$ (the difference of two monotone increasing functions) is differentiable, with derivative~$J_1\pr$, and that 
  $$f\in{\cal F}_{J_1}:=\{f\in{\cal F}_0\vert \ \lim_{x\to\pm\infty}J_1(F(x))f(x) = 0\},
  $$
so that (\ref{AREf3}) holds. Finally, assume that   $J_1$ is  
    skew-symmetric about $1/2$.  
Defining the (possibly infinite) constants $$\kappa^{+}_J := \sup_{u
   \ge 1/2} \left| J\pr(u)  \right| \quad \mbox{ and }\quad
\kappa^{-}_J := \inf_{u \ge 1/2} \left| J\pr(u) \right|,$$ 
we
can always write  
\begin{equation}
  \label{eq:2}
  12 \,  \text{ARE}_{f}\left(\phi_{J_1}/\phi_{{\rm W}}\right)  \le
(\kappa^{+}_{J_1})^2/ \mathcal{K}(J_1)
\end{equation}
while, if  $J_1$ is non-decreasing (hence $J_1\pr$ is non-negative),  we further  have 
\begin{equation}
  \label{eq:1}
(\kappa^{-}_{J_1})^2/ \mathcal{K}(J_1) \le   12 \,  \text{ARE}_{f}\left(\phi_{J_1}/\phi_{{\rm W}}\right)  \le
(\kappa^{+}_{J_1})^2/ \mathcal{K}(J_1).
\end{equation}
The  quantities appearing in  \eqref{eq:2} and \eqref{eq:1} often   can be
computed explicitly,  yielding   
ARE bounds which are, moreover, sharp under certain conditions (see below). 

For example, if $J_1$ is convex on $[1/2, 1)$, 
 its  derivative $J_1\pr$ is non-decreasing
over~$[1/2,1)$, so that 
\begin{equation}\label{kappapm}
\kappa^{-}_{J_1}  = J_1\pr(1/2)\geq 0 \quad \text{ and }\quad
\kappa^{+}_{J_1}  =\lim_{u\to1}J_1\pr(u)\le +\infty .
\end{equation}  
 It follows  that,  under the  assumptions made, 
\begin{equation}
  \label{eq:convex}
(J_1\pr(1/2))^2/\mathcal{K}(J_1) \le   12\,  \text{ARE}_{f}\left(\phi_{J_1}/\phi_{{\rm W}}\right)
\le    (\lim_{u\to1}J_1\pr(u))^2/\mathcal{K}(J_1).
\end{equation}
The lower bound in \eqref{eq:convex} is established in  Theorem
  2.1 of  \cite{G70}. 
  
  The double inequality~(\ref{eq:convex}) holds, for instance (still, under $f\in{\cal F}_{J_1}$), when
the scores $J_1 = \varphi_g\circ G^{-1}$ are the optimal scores
associated with some symmetric and  {\it strongly unimodal} density~$g$ with distribution function $G$; such densities indeed are log-concave and have monotone increasing,  convex over $[1/2, 1)$ score functions.  
  Symmetric log-concave densities take the form  
\begin{equation}
  \label{eq:1+}
  g(x) = K e^{-\mu(x)}, \qquad K^{-1}=\int_{-\infty}^\infty e^{-\mu(x)}{\rm d}x
\end{equation}
with  
 $x \mapsto \mu(x)$ a convex, even  (that is,
$\mu(x) = \mu(-x)$) function; assume it to be twice differentiable,  with derivatives $\mu\pr$ and $\mu\ppr$.  
Then, 
$  \varphi_g(x) = \mu\pr(x)$, so that 
$$  J_1(u) := \varphi_g(G^{-1}(u))= \mu\pr(G^{-1}(u)), \quad \quad\!\!\!  \mathcal{K}(J_1) \! =\!\! \int 
_{-\infty}^\infty \big(\mu\pr(x)\big)^2g(x){\rm d}x= \mathcal{I}(g)$$
where $\mathcal{I}(g)$ the Fisher information of $g$ (which we assume to be finite), and 
$$  J_1\pr(u) =  {{\mu\ppr}(G^{-1}(u))}/{g(G^{-1}(u))}, \quad\text{hence}\quad  J_1\pr(1/2) = \frac{{\mu\ppr}(0)}{g(0)} = \frac{{\mu\ppr}(0)}{K}.$$
Specializing   \eqref{eq:convex} to this situation, we obtain 
the following proposition. 
\begin{proposition}
  If  the square-integrable score-generating function $J_1$ is of the form $ \varphi_g\circ G^{-1}$ with $g$ given by \eqref{eq:1+}, $\mu$ even, convex, and  twice differentiable,   then,   under any $f\in{\cal F}_{J_1}$,   
\begin{equation}
  \label{eq:7}
 \left( \frac{{\mu\ppr}(0)}{K}\right)^2  \le 12\,  \mathcal{I}(g) ARE_f
  (\phi_{J_1}/\phi_{\rm W}) \le  ( \lim_{u\to1}J_1\pr(u))^2  =  ({\lim_{x\to \infty}({\mu\ppr}(x)}/ g(x))^2.
\end{equation}
\end{proposition}

With $\mu(x)= x^2/2$  (so that $K^{-1} = \sqrt{2\pi}$) in (\ref{eq:1+}), $g$ is the standard Gaussian density;  $\mu\ppr(0) = 1$, $\mathcal{I}(g)=1$, and the lower bound in (\ref{eq:7}
) becomes $(\mu\ppr(0)/K)^2 = 2\pi$, whereas the upper bound is trivially infinite. This yields the Hodges-Lehmann result (\ref{eqbis}). \vspace{2mm}

Turning back to \eqref{eq:2} and \eqref{eq:1}, but with  $J_1$   concave (and still non-decreasing) on $[1/2,
1)$,    $J_1\pr$ is nonincreasing, so that $\kappa^+_{J_1}   = J_1\pr(1/2)$ and
\begin{equation}
  \label{eq:concare}
 12\,  \text{ARE}_{f}\left(\phi_{J_1}/\phi_{{\rm W}}\right)
\le   ( J_1\pr(1/2))^2/\mathcal{K}(J_1).
\end{equation}
Not  much can be said on the lower bound, though, without
  further assumptions on the behavior of $J_1$ around $u=1$.

Replacing, for various score-generating functions $J_1$ and densities
$f$, the quantities appearing in (\ref{eq:2}), (\ref{eq:convex}) or
(\ref{eq:concare}) with their explicit values provides a variety of
bounds of the Hodges-Lehmann type. Below, we
consider  the van der Waerden tests $\phi_{{\rm vdW}}$,   
based on the
score-generating function  \eqref{eq:3}
 and the Cauchy-score rank  tests  $\phi_{{\rm Cauchy}}$, based on the
score-generating function 
\begin{equation}\label{scCau}
J_{\rm Cauchy}(u)=\sin (2 \pi (u- 
 {1}/{2})). 
  \end{equation}
  
\begin{figure}[htbp]
\begin{center}
\includegraphics[height=118mm, width=\linewidth]
{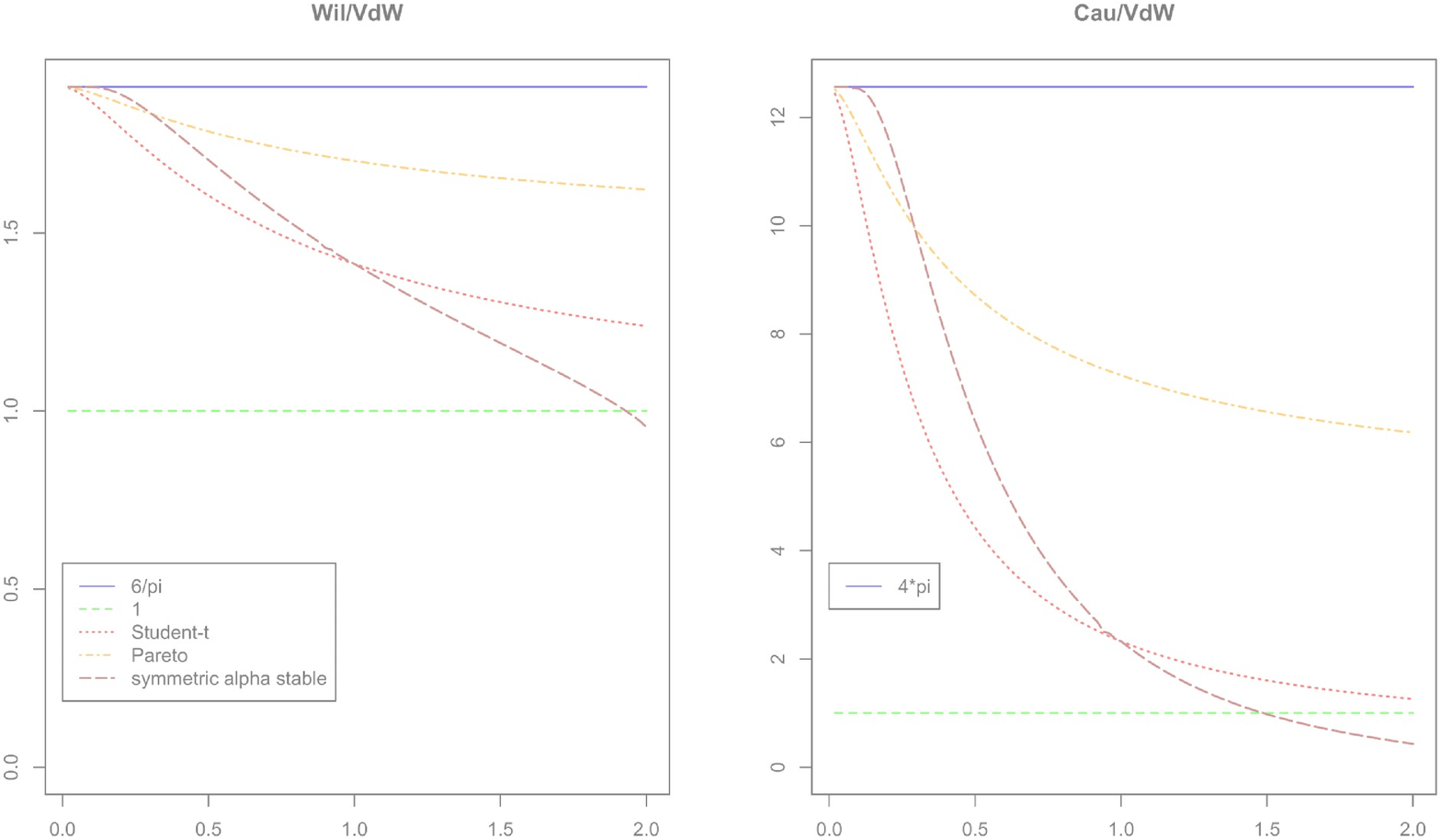}
\caption{\small ${\rm \text{ARE}}_f(\phi_{{\rm W}}/\phi_{{\rm vdW}})$ and ${\rm \text{ARE}}_f(\phi_{{\rm Cauchy}}/\phi_{{\rm vdW}})$ under various families of densities:  symmetric stable (indexed by their tail parameter $\alpha$), Student-$t$ (indexed by their degrees of freedom $\nu$) or Pareto (indexed by their shape parameter $\alpha$). 
}
\label{distD}
\end{center}
\end{figure}

\begin{proposition} \label{bounds}
 For all symmetric  
  densities $f$ in ${\cal F}_{{\rm vdW}}$,  ${\cal F}_{{\rm Cauchy}}$ and ${\cal F}_{{\rm vdW}}\bigcap {\cal F}_{{\rm Cauchy}}$, respectively, 
  \begin{itemize}
  \item[(i)]  ${\rm \text{ARE}}_f(\phi_{{\rm W}}/\phi_{{\rm vdW}}) \leq  {6}/{\pi}$;
  \item[(ii)] ${\rm \text{ARE}}_f(\phi_{{\rm Cauchy}}/\phi_{{\rm W}}) \leq   {2\pi^2}/{3}$;
  \item[(iii)] ${\rm \text{ARE}}_f(\phi_{{\rm Cauchy}}/\phi_{{\rm vdW}}) \leq 4 \pi$.
  \end{itemize} 
 \end{proposition}

\begin{proof}
The van der Waerden score \eqref{eq:3} is strictly increasing, and  convex  over $[1/2,1)$. One readily obtains
$$\mathcal{K}(J_{\rm vdW})=1\quad\text{ and }\quad  J\pr_{\rm vdW}(u)=
\sqrt{2\pi}\exp \{(\Phi^{-1}(u))^2/2\}, 
$$  
hence $ \kappa^-_{{\rm vdW}}  =  J\pr_{\rm vdW}(1/2)= \sqrt{2\pi}$.  
Plugging this into the left-hand side inequality of  \eqref{eq:convex}
yields   (i).   Alternatively one can directly apply \eqref{eq:7}.

The Cauchy
score is  concave over $[1/2, 1)$, but not monotone  (being of bounded variation, however, it is the difference of two monotone function). Direct inspection of  \eqref{scCau} nevertheless reveals that 
$$\mathcal{K}(J_{{\rm Cauchy}})=  1/2 \quad\text{ and }\quad J\pr_{{\rm Cauchy}}(u)= 2\pi\cos (2\pi(u-1/2)),$$
hence   $  \kappa^+_{{{\rm Cauchy}}} = J\pr_{{\rm Cauchy}}(1/2)= {2\pi}$. Substituting this in (\ref{eq:2}) yields (ii). The product of the upper bounds in (i) and (ii) yields (iii). 
\end{proof}

  Remarkably,   those three bounds are sharp.  Indeed,  
 numerical evaluation shows that they can be approached
arbitrarily well by taking extremely heavy-tails  such as those of
stable densities $f_\alpha$ with tail index~$\alpha\to 0$,  Student densities
with degrees of freedom~$\nu\to 0$, or Pareto densities with
$\alpha\to 0$; see  also the family ${\cal F}_{\rm HL}$ of densities $f_{a, \epsilon}(x)$Ê defined inÊ equation (\ref{eq:2d}).  

Figure \ref{distD} provides plots of   ${\rm \text{ARE}}_f(\phi_{{\rm
    W}}/\phi_{{\rm vdW}})$ and ${\rm \text{ARE}}_f(\phi_{{\rm
    Cauchy}}/\phi_{{\rm vdW}})$  for  various densities.   Inspection of
those graphs shows that both AREs are decreasing as the tails become
lighter; the sharpness of bounds (i) and (iii), hence also that of
bound (ii), is graphically confirmed.

The bounds proposed in Proposition~\ref{bounds} are not new, and  
 have been obtained  already 
 in \cite{G70}. One would like to see similar bounds for other score functions, such as the  Student  ones
\begin{eqnarray} J_{t_\nu}(u)
&=& {(\nu +1)
  F^{-1}_{t_\nu}(u)}/({\nu + F^{-1}_{t_\nu}(u)^2)} \qquad\qquad\quad   0 < u < 1\nonumber \\ 
  &=& 
\frac{1+\nu }{\sqrt{\nu}}
\sqrt{-1+\frac{1}{\text{IB}_\nu(1-2u)}}\text{IB}_\nu(1-2u)\  \  1/2\leq u < 1
\label{Studscore}\end{eqnarray}
 where $\mbox{IB}_\nu (v)$ denotes  the   inverse of the regularized
incomplete beta function evaluated at $(1, v, \nu/2, 1/2)$ and  $F^{-1}_{t_\nu}$ stands for the Student 
quantile function with $\nu$ degrees of freedom. Note that $\lim_{v\to -1}\mbox{IB}_\nu (v) = 0$, so that $\lim_{u\to 1}J_{t_\nu}(u)=0$. Since $J_{t_\nu}(1/2)=0$ and $J\pr _{t_\nu}(1/2)>0$, this means that, on $[1/2, 1)$,  $J_{t_\nu}$ is a redescending function; in general, it is neither convex nor concave on $[1/2,1)$.

Differentiating (\ref{Studscore}), we get, for $u \ge 1/2$,  
\begin{equation}\label{Yvik}
J_{t_{\nu}}'(u) =   \frac{\sqrt{\pi } (\nu +1 )
  \Gamma\left(\frac{\nu }{2}\right)}{\sqrt{\nu }
  \Gamma\left(\frac{\nu +1 }{2}\right)}  
 \left(-1+2 
\text{IB}_\nu(1-2u)\right)\text{IB}_\nu(1-2u)^{\frac{1-\nu}{2}},
\end{equation}
from which we deduce that 
$$\lim_{u\to 1}J\pr_{t_\nu}(u)=\left\{
\begin{array}{rc} 0& \quad 0<\nu <1\\ -2\pi& \quad \nu = 1 \\ -\infty& \quad 1<\nu\phantom{<1} .
\end{array}\right.
$$
Except for the $\nu=1$ case, which   is covered by (ii) and  (iii) in Proposition~\ref{bounds}, these values do not provide exploitable values for $\kappa^+$. 
For $\nu<1$, however, one can check from (\ref{Yvik})   that 
$\max_{u\geq 1/2} \vert J\pr(x)\vert = J\pr(1/2)$, so that 
$$\kappa^+_{J_{t_\nu}}
=- {\sqrt{\pi}(\nu +1) \Gamma\left(\frac{\nu}{2}\right)}{\Big /}{\sqrt{\nu}\,\Gamma\left(\frac{\nu +1}{2}\right)}.$$ 
Elementary though somewhat tedious algebra yields  $$\mathcal{K}(J_{{t_\nu}})=  (\nu +1)/(\nu + 3).$$
Plugging this into (\ref{eq:2}), we obtain, for~$\nu \leq1$, the following additional bounds. 
\begin{proposition} \label{bounds'}
 For all  $0<\nu\le1$ and all symmetric density $f$ in ${\cal F}_{J_{t_\nu}}$ and ${\cal F}_{J_{t_\nu}}\bigcap {\cal F}_{J_{{{\rm vdW}}}}$, respectively,   
  \begin{itemize}
\item[(iv)] ${\rm \text{ARE}}_f(\phi_{t_\nu}/\phi_{{\rm W}}) \leq  {\pi
      \Gamma^2(\frac{\nu}{2})(\nu +3)(\nu+1)}/{12 \nu \Gamma^2(\frac{\nu +1}{2})}$, and 
\item[(v)] ${\rm \text{ARE}}_f(\phi_{t_\nu}/\phi_{{\rm vdW}}) \leq {\Gamma^2(\frac{\nu}{2})(\nu +3)(\nu+1)}/{2 \nu \Gamma^2(\frac{\nu +1}{2})}$. 
\end{itemize}
\end{proposition}  
\noindent Inequality (iv) is sharp, the bound being achieved, in the
limit, under very heavy tails (stable densities with $\alpha\downarrow
0$, or Student-$t_\mu$ densities with $\mu\downarrow 0$). Since this
is also the case, under the same sequences of densities,  for
inequality (i) in Proposition~2.1, inequality (v) is sharp as well.
  The upper bounds (iv) and (v) are both decreasing
  functions of the tail index $\nu$; both are unbounded at the 
  origin, and both   converge to the corresponding Cauchy values as 
  $\nu\mapsto 1$.

\subsection{The serial case} \label{sercase}

Until the early eighties, and despite  some forerunning  time-series applications such as Wald and Wolfowitz~\cite{WW}  (published as early as 1943---two years before Frank Wilcoxon's pathbreaking 1945  paper \cite{FW}!),  rank-based methods had been  essentially limited to statistical models involving
univariate independent observations. Therefore,  the traditional ARE bounds (Hodges and 
Lehmann \cite{HL56,HL61},  Chernoff-Savage~\cite{CS58} or Gastwirth \cite{G70}), as well as the classical monographs (H\'ajek and  \v{S}id\a'ak~\cite{HvS},  Randles and Wolfe~\cite{RW79}, Puri and Sen~\cite{PS85}, to quote only a few) 
mainly deal with univariate location and single-output linear (regression) 
models with independent observations. 
The situation  since then has changed, and 
rank-based procedures  nowadays have been proposed for a much broader class of statistical models, including time series problems, where serial dependencies    are  the main features under study. 

In this section, we focus on the linear rank statistics of the serial type involving two square-integrable score functions. Those statistics enjoy optimality properties  in the context of linear time series   (ARMA models; see \cite{HP94} for details).  Once adequately standardized, those statistics yield the so-called {\it rank-based autocorrelation coefficients}.   Denote  by $R\n_{1}, \ldots , R\n_{n}$ the ranks in  a triangular array~$X\n_{1}, \ldots , X\n_{n}$ of observations. {\it Rank  autocorrelations} (with lag $k$) are linear  serial
rank statistics  of the form 
$$
{\ut r}^{\!\!\!(n)}_{\!\!\! \raisebox{1.3ex}{$\scriptstyle{J_1J_2;k}$}}  :=
 \Big[(n-k)^{-1}\sum
_{t=k+1}^{n}J_1\Big(\frac{R\n_{t}}{n+1}\Big)J_2\Big(\frac{R\n_{t-k}}{n+1}\Big)
- m\n_{J_1J_2}\Big]  \big(s\n_{J_1J_2}\big)^{-1} ,
$$
where $J_1$ and $J_2$ are (square-integrable) score-generating  functions, whereas $m\n_{J_1J_2}$ and $s\n_{J_1J_2} :=
s\n_{J_1J_2;k}$ denote the exact mean of
$
J_1\Big(\frac{R\n_{t}}{n+1}\Big) J_2\Big(\frac{R\n_{t-k}}{n+1}\Big)$ and the exact
standard error  of   $(n-k)^{-\frac{1}{2}}\! \sum
_{t=k+1}^{n}\! J_1\Big(\frac{R\n_{t}}{n+1}\Big)J_2\Big(\frac{R\n_{t-k}}{n+1}\Big)\vspace{-0.7mm}$
under the assumption of i.i.d.\ $X\n_t$'s (more precisely, exchangeable $R_t\n$'s), respectively;  we refer to   pages~186 and 187 of \cite{HP94} for
explicit formulas. {\it Signed-rank autocorrelation coefficients} are
defined  similarly; see \cite{HP92} or \cite{HP94}.

Rank and  signed-rank autocorrelations are measures of serial dependence 
 offering  rank-based alternatives to the usual autocorrelation
coefficients, of the form 
$$r\n_{k}:=\sum _{t=k+1}^{n} X_tX_{t-k}/ \sum _{t=1}^{n} X_t^2,$$
which consitute the Gaussian reference benchmark in this context.  Of particular interest
are \vspace{-1mm}
\begin{enumerate}
\item[(i)] the {\it van der Waerden autocorrelations}  \cite{HP88}\vspace{-3mm}
$${\ut r}^{\!\!\!(n)}_{\!\!\! \raisebox{1.3ex}{$\scriptstyle{{\text{vdW}};k}$}}  :=
\Big[(n-k)^{-1}\sum_{t=k+1}^{n}\Phi ^{-1}\Big(\frac{R\n_{t}}{n+1}\Big)
\Phi ^{-1}\Big(\frac{R\n_{t-k}}{n+1}\Big) - m\n_{{\text{vdW}}}\Big]\! 
\big( s\n_{{\text{vdW}}}\big)^{-1}\!\!\! ,
$$\vspace{-3mm}
\item[(ii)] the {\it Wald-Wolfowitz} or {\it  Spearman autocorrelations} \cite{WW}\vspace{-3mm}
$${\ut r}^{\!\!\!(n)}_{\!\!\! \raisebox{1.3ex}{$\scriptstyle{{\text{SWW}};k}$}} := 
\Big[(n-k)^{-1}\sum_{t=k+1}^{n}R\n_{t}
R\n_{t-k}- m\n_{{\text{SWW}}}\Big]
\big(s\n_{{\text{SWW}}}\big)^{-1}, $$ \vspace{-2mm}
\item[(iii)] and the {\it Kendall autocorrelations}  \cite{FGH} (where explicit
values of $m\n_{{\text{K}}}$ and $s\n_{{\text{K}}}$ are provided)\vspace{-3mm}
$$
{\ut r}^{\!\!\!(n)}_{\!\!\! \raisebox{1.3ex}{$\scriptstyle{{\text{K}};k}$}} :=\Big[ 1 -
\frac{4D\n_k}{(n-k)(n-k-1)} -m\n_{{\text{K}}}\Big]\big(s\n_{{\text{K}}}\big)^{-1} \vspace{-3mm}
$$
with $D\n_k$ denoting the number of discordances at lag $k$, that is, the number of pairs 
$(R\n_t,R\n_{t-k})$ and $(R\n_s,R\n_{s-k})$ that satisfy either 
$$R\n_t<R\n_s\quad\text{and}\quad 
R\n_{t-k}>R\n_{s-k},\quad\!\text{or}\!\quad R\n_t>R\n_s\quad\!\text{and}\!\quad R\n_{t-k}<R\n_{s-k};$$ more specifically, 
  $
D\n_k :=\sum_{t=k+1}^n\sum_{s=t+1}^n I(R\n_t<R\n_s,\, R\n_{t-k}>R\n_{s-k}).$ 
 \vspace{-1mm}
\end{enumerate}
The van der Waerden autocorrelations are optimal---in the
sense that they allow for {\it locally optimal} rank tests in the case of ARMA models with normal innovation densities.
The Spearman and Kendall autocorrelations are serial versions of
Spearman's {\it rho} and Ken\-dall's {\it tau}, respectively, and are asymptotically equivalent under the null hypothesis of independence; although they are never optimal for any ARMA alternative,  they achieve excellent overall performance.
Signed rank autocorrelations are defined in a similar way.

 Let $J_i$, $i=1,\ldots , 4$ denote four square-summable score functions, and assume that they are monotone increasing, or the difference between two monotone increasing functions (that assumption  tacitly will be made in the sequel each time AREs are to be computed). Recall that  ${\cal F}_2$ denotes the subclass of densities~$f\in\cal F$ having finite moments of order two. The asymptotic relative efficiency, under innovation   density $f\in{\cal F}_2$, of the  rank-based tests $\phi^r_{J_1J_2}$  
 based on the autocorrelations${\ut r}^{\!\!\!(n)}_{\!\!\! \raisebox{1.3ex}{$\scriptstyle{J_1J_2;k}$}}\vspace{-2mm}$ with  respect to the  rank-based tests~$\phi^r_{J_3J_4}$ based on the autocorrelations${\ut r}^{\!\!\!(n)}_{\!\!\! \raisebox{1.3ex}{$\scriptstyle{J_3J_4;k}$}}\vspace{-2mm}$    is  
\begin{align}
& \text{ARE}^*_f(\phi^r_{J_1 J_2}/\phi^r_{J_3 J_4})  \nonumber \\
& =   
  \frac{ \mathcal{K}(J_3)}{\mathcal{K}(J_1)} \left(   \frac{\int_{0}^{1} J_1(v)
    \varphi_{f}(F^{-1}(v)){\rm d}v}{\int_{0}^{1} 
    J_3(v)\varphi_{f}(F^{-1}(v)) {\rm d}v}  \right)^2
     \frac{\mathcal{K}(J_4) }{\mathcal{K}(J_2)}
 \left(\frac{\int_0^1 
   J_2(v) F^{-1}(v) {\rm d}v}{\int_0^1 J_4(v) 
    F^{-1}(v) {\rm d}v} \right)^2 \nonumber \vspace{1mm} \\
& =\frac{ \mathcal{K}(J_3)}{\mathcal{K}(J_1)} C^2_f(J_1, J_3) \frac{ \mathcal{K}(J_4)}{\mathcal{K}(J_2)} \, D^2_f(J_2,J_4)   \label{eq:9}  \vspace{1mm}
\end{align}
with $C_f(J_1, J_3) := {\cal K}(J_1,f)/{\cal K}(J_3,f)$ and $D_f(J_2, J_4):= {\cal J}(J_2,f)/{\cal J}(J_4,f)$.

The $C_f$ ratios have been studied in Section~\ref{location},  and the same conclusions apply here; as for
the $D_f$  ratios, they can be treated by similar methods.

Denote by $\phi_{{\rm vdW}}^r$, $\phi_{{\rm W}}^r$, $\phi_{{\rm SWW}}^r$, $\ldots$ the tests
based on${\ut r}^{\!\!\!(n)}_{\!\!\!
  \raisebox{1.3ex}{$\scriptstyle{{\text{vdW}};k}$}}$,
  ${\ut
  r}^{\!\!\!(n)}_{\!\!\!
  \raisebox{1.3ex}{$\scriptstyle{{\text{W}};k}$}} $,
  ${\ut
  r}^{\!\!\!(n)}_{\!\!\!
  \raisebox{1.3ex}{$\scriptstyle{{\text{SWW}};k}$}} \vspace{-2mm} $,  etc. The serial counterpart of $\text{ARE}_f(\phi_{{\rm W}}/\phi_{J_1})$ is $\text{ARE}^*_f(\phi_{{\rm SWW}}^r/\phi_{J_1J_2}^r)$, for which the following result holds. 
  
 \begin{proposition} \label{serbounds}   Let the score functions $J_1$ and $J_2$ be monotone increasing, 
  skew-symmetric about $1/2$, and differentiable, with  strictly positive $J\pr_1(1/2)$ and~$J\pr_2(1/2)$.  
 Suppose that $f\in{\cal F}_2\bigcap{\cal F}_{J_1}\bigcap{\cal F}_{J_2}$ is a symmetric probability density 
function.  
 Then, if $J_1$ and $J_2$ are  \begin{itemize}
\item[(i)]  convex on 
$[1/2, 1)$,  
$$
\text{ARE}^*_f(\phi_{{\rm SWW}}^r/\phi_{J_1J_2}^r)
=
\text{ARE}^*_f(\phi_{{\rm K}}^r/\phi_{J_1J_2}^r) \leq
144\frac{\mathcal{K}(J_1) \mathcal{K}(J_2)}{(J_1\pr (1/2)\ J_2\pr (1/2))^2};
$$
\item[(ii)]  concave on $[1/2, 1)$,  
$$
\text{ARE}^*_f(\phi_{J_1 J_2}^r/\phi_{{\rm SWW}}^r)
=
\text{ARE}^*_f(\phi_{J_1 J_2}^r/\phi_{{\rm K}}^r) \leq  
\frac{1}{144}\frac{(J_1\pr (1/2)\ J_2\pr (1/2))^2}{\mathcal{K}(J_1) \mathcal{K}(J_2)}.
$$
\end{itemize}
\end{proposition}

\begin{proof} In view of (\ref{AREf}), we have
$$\text{ARE}^*_f(\phi_{{\rm SWW}}^r/\phi_{J_1J_2}^r)
=\text{ARE}_f(\phi_{{\rm W}}/\phi_{J_1})\frac{{\cal K}(J_2)}{{\cal K}(J_W)}\left(\frac{\int_0^1(v-1/2)F^{-1}(v){\rm d}v}{\int_0^1J_2(v)F^{-1}(v){\rm d}v}\right)^2.
$$
Consider part (i) of the proposition. It follows from \eqref{eq:1} 
 that 
$$\text{ARE}_f(\phi_{{\rm W}}/\phi_{J_1})\leq  {12\, {\cal K}(J_1)}/{(J\pr_1(1/2))^2}.
$$
Since $J_2$ is convex over $[1/2, 1)$,   $J_2(u) \ge J_2\pr (1/2)(u-1/2)$ for
  all~$u \in [1/2, 1)$, so that  
$$\int_{0}^1 J_2(v) 
    F^{-1}(v) {\rm d}v =
  2  \int_{1/2}^1 J_2(v) 
    F^{-1}(v) {\rm d}v \ge J_2\pr (1/2)\int_{1/2}^1 (v-1/2) 
    F^{-1}(v) {\rm d}v .$$
It follows that 
$$\frac{{\cal K}(J_2)}{{\cal K}(J_W)}\left(\frac{\int_0^1(v-1/2)F^{-1}(v){\rm d}v}{\int_0^1J_2(v)F^{-1}(v){\rm d}v}\right)^2
\leq \frac{12\, {\cal K}(J_2)}{(J\pr_2(1/2))^2},
$$
where the assumption of finite variance is used. Part (i) of the
result follows.   A similar argument  holds (with 
reversed inequalities) if~$J_2$ is concave, yielding part~(ii).  \end{proof}

 Applying this result 
to  the  score functions  $J_1(u) =
  J_2(u) = \Phi^{-1}(u)$ (convex over $[1/2,0)$)  for which $J_1\pr (1/2) = J_2\pr (1/2) =
  \sqrt{2\pi}$ and $\mathcal{K}(J_1) = \mathcal{K}(J_2) = 1$,  we readily 
  obtain the following serial 
  extension of Hodges and  Lehmann's  ``$6/\pi$ result":
  \begin{equation}
  \label{eq:15}
\text{ARE}^*_f(\phi_{\rm SWW}^r/\phi_{\rm vdW}^r)
=  \text{ARE}^*_f(\phi_{\rm K}^r/\phi_{\rm vdW}^r)
   \le (6/\pi)^2.
\end{equation}

An important difference, though, is that  the bound in  \eqref{eq:15}  is unlikely  to be sharp.  Section~\ref{sec:numer-cons}  provides some  numerical evidence of that fact, which is hardly surprising: 
  while the ratio $C_f(J_{\rm vdW}, J_{\rm W})$ is maximized for
  densities putting all their weight 
about the origin, this no longer holds true for $D_f(J_{\rm vdW},
J_{\rm W})$.  
 In particular, the sequences of densities considered  in \cite{HL61} or 
\cite{G70} along which $C_f(J_{\rm vdW},
J_{\rm W})$ tends to its upper bound  typically are not the same as those along which  
$D_f(J_{\rm vdW},
J_{\rm W})$ does. %

\setcounter{equation}{0}
\section{Some numerical results}
\label{sec:numer-cons}
In this final section, we provide numerical values of   $ARE_{f}(\phi_{\rm W}/\phi_{\rm vdW})$ (denoted as $ARE_f$ in the sequel) and $ARE^*_{f}(\phi^r_{\rm SWW}/\phi^r_{\rm vdW})$ (denoted as $ARE_f^*$ in the sequel) under various families of distributions.

 First, let us give some ARE values under  Gaussian densities: if $f=\phi$, we obtain 
 \begin{equation*}
    C_\phi(J_{\rm W},J_{\rm vdW})
     =
    D_\phi(J_{\rm W},J_{\rm vdW}) =  \frac{1}{2\sqrt{\pi}} \approx 0.28209
  \end{equation*}
so that 
\begin{equation*}
  ARE_{\phi}(\phi_{\rm W}/\phi_{\rm vdW}) = \frac{3}{\pi} \approx 0.95493
\end{equation*}
and 
\begin{equation*}
  ARE^*_{\phi}(\phi^r_{\rm SWW}/\phi^r_{\rm vdW}) = \frac{9}{\pi^2}
  \approx 0.91189.
\end{equation*}

\begin{table}
\begin{center}
    \begin{tabular}{|c|cccc|}
\hline
$\epsilon$ & $C_f$ & $D_f$ & $ARE_f$ & $ARE_f^*$  \\
\hline
\hline
0 & .398942 & .282070 &   1.90986 &  1.82346 \\
.2 & .396313 & .276619 & 1.88476  &  1.73062 \\
.4 &  .388772 &  .271848 &  1.81372 & 1.60844  \\
.6 & .377291 & .271061 &  1.70818  & 1.50608 \\
1 &  .348213 &   .287973 & 1.45503 &  1.44796 \\
 2 &.294160 & .303085 & 1.03836 & 1.14461 \\
3& .282852 & .285646 & .960064 &  .940023 \\
10 & .282095 &   .282095 & .954930 & .911891 \\
 100 &   .282095 &  .282095 & .954930 & .911891 \\
\hline
    \end{tabular}\vspace{-1mm}
    \caption{\label{hodges}Numerical values of $C_f$,  $D_f$,
      $ARE_f=ARE_{f}(\phi_{\rm W}/\phi_{\rm vdW})$ and $ARE_f^*=ARE^*_{f}(\phi^r_{\rm SWW}/\phi^r_{\rm vdW})$ under  densities   $f_{a,\epsilon}$  in the Hodges-Lehmann family ${\cal F}_{\rm HL} $ ( see (\ref{eq:2d})), for
      various values of  $\epsilon$ and $a \to 0$.} 
  \end{center}
  \end{table}
  \begin{table}
  \label{stud}
  \begin{center}
    \begin{tabular}{|c|cccc|}
\hline
$\nu$ & $C_f$ & $D_f$ & $ARE_f$ & $ARE_f^*$  \\
\hline
\hline
0.1 &  .394451 & -- &   1.86710 &  -- \\
1 & .343120 & -- & 1.41277  &  -- \\
2 &  .321212 &  .243196 &  1.23813 & .878736  \\
4 & .304695 & .269173 &  1.11407   &.968623  \\
6 & .297953 &    .274541 & 1.06531 &  .963551 \\
8 &  .294303 &   .276784 & 1.03937 &   .955507 \\
 10 & .292017 & .278005 &1.02329 & .949042 \\
100& .283146 & .281737 & .962059 &  .916370 \\
\hline
   \end{tabular}\vspace{-1mm}
    \caption{\label{stud} Numerical values of $C_f$,  $D_f$, $ARE_f=ARE_{f}(\phi_{\rm W}/\phi_{\rm vdW})$
      and $ARE_f^*=ARE^*_{f}(\phi^r_{\rm SWW}/\phi^r_{\rm vdW})$ under  Student-$t$ densities with various 
      degrees of freedom $\nu$.} 
  \end{center}
  \end{table}
\begin{table}
\label{poexp}
\begin{center}
    \begin{tabular}{|c|cccc|}
\hline
$\alpha$ & $C_f$ & $D_f$ & $ARE_f$ & $ARE_f^*$  \\
\hline
\hline
0.1 & .393903 & .175222 & 1.86191  &  0.685991 \\
1 & .313329 & .2720600 & 1.1781  &  1.046388 \\
2 & .282095 & .2820950  &  .954930 & .911893   \\
10 & .222095 & .2934363  & .591916    &  .611600\\
100 & .168549 & .2953577   & .340904 & .356871  \\
\hline
    \end{tabular}\vspace{-1mm}
    \caption{\label{poexp} Numerical values of $C_f$,  $D_f$, $ARE_f=ARE_{f}(\phi_{\rm W}/\phi_{\rm vdW})$
      and $ARE_f^*=ARE^*_{f}(\phi^r_{\rm SWW}/\phi^r_{\rm vdW})$ under  power exponential densities for 
     various values of the shape parameters $\alpha$.} 
  \end{center}
  \end{table}

Tables 1-3   provide numerical values of $ARE_f$ and  $ARE_f^*$ under 
\begin{enumerate}
\item[(i)] (Table 1) the two-parameter  family ${\cal F}_{\rm HL} $ of densities  $f_{a,\epsilon}$ associated with the distribution functions 
\begin{equation}
  \label{eq:2d}
  F_{a,\epsilon}(x) = \left\{
    \begin{array}{cl}
      \Phi (x)& \mbox{ if } \ 0 \le x \le \epsilon\\
       \Phi (\epsilon  + a(x-\epsilon))& \mbox{ if }\ \epsilon < x
    \end{array}
\right.
\end{equation}
where   $F_{a, \epsilon}(x)$ is
defined by symmetry for $x \le 0$ (this family of distributions, which has been used by Hodges and Lehmann  \cite{HL61}, is such that the nonserial $6/\pi $ bound is 
  achieved, in the limit,    
as both $a$ and $\epsilon$ go to zero),
\item[(ii)] (Table 2) the   family ${\cal F}_{\rm Student} $ of Student densities with degrees of freedom $\nu >0$, and
\item[(iii)]  (Table 3) the family ${\cal F}_{e}$ of power-exponential densities, of the form 
  \begin{equation}
    \label{eq:powexp}
    f_{\alpha}(x)  := \frac{e^{-|x|^{\alpha}}}{2\Gamma(1+1/\alpha)}\qquad x\in\mathbb{R},\ \ \alpha >0.
  \end{equation}
\end{enumerate}

\begin{figure}[htbp]
\begin{center}
\includegraphics[height=118mm, width=\linewidth]
{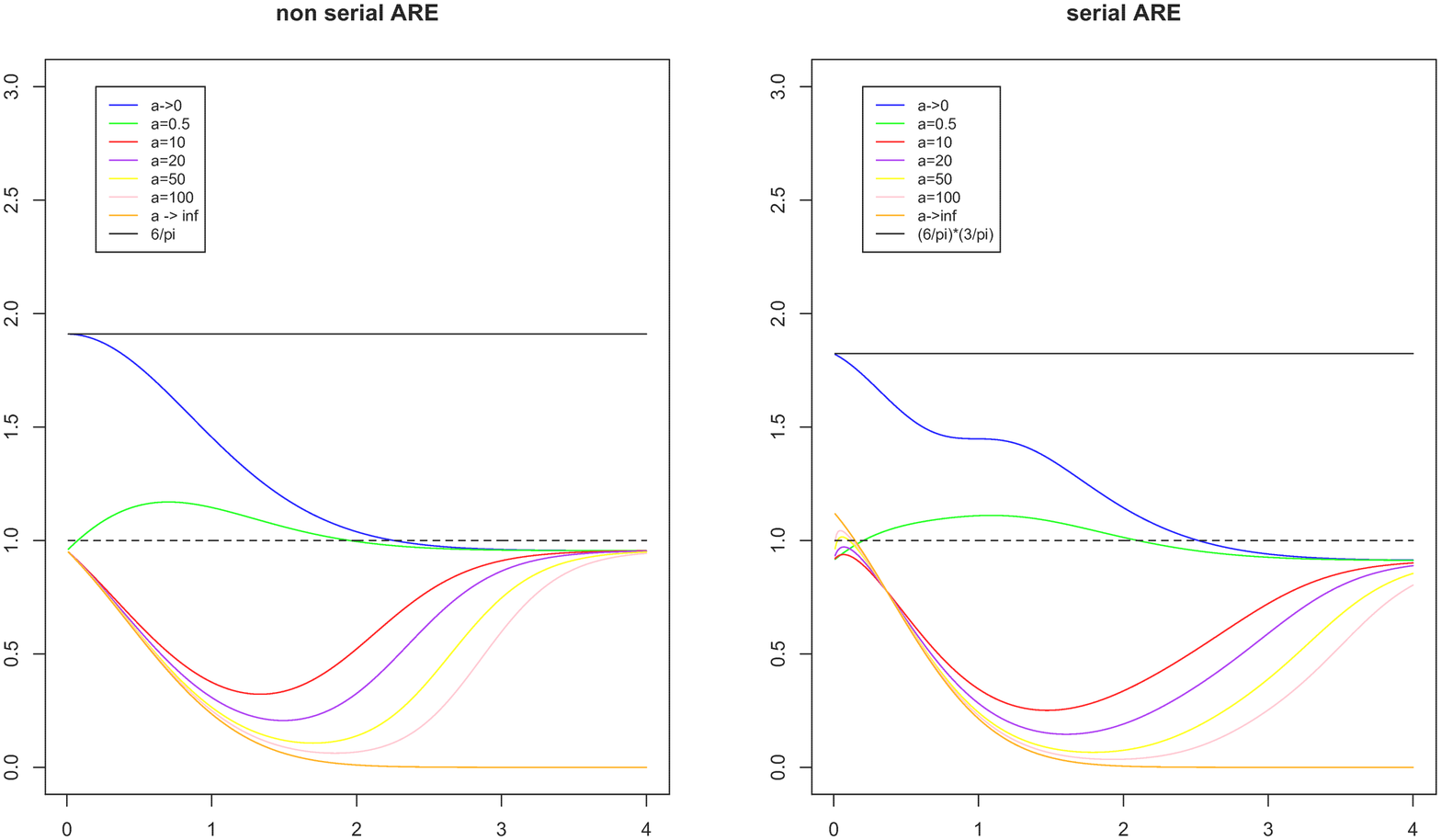}
\caption{\small Nonserial ${\rm \text{ARE}}_f=ARE_{f}(\phi_{\rm W}/\phi_{\rm vdW})$ (left plot) and serial  ${\rm
    \text{ARE}}^\star_f=ARE^*_{f}(\phi^r_{\rm SWW}/\phi^r_{\rm vdW})$ (right plot) 
  under  densities   $f_{a,\epsilon}$  in the Hodges-Lehmann family ${\cal F}_{\rm HL} $ ( see (\ref{eq:2d})),  as a function of $\epsilon \in [0, 4]]$, for various choices of  the parameter $a$.
} 
\label{HLplot}
\end{center}
\end{figure}

\begin{figure}[htbp]
\begin{center}
\includegraphics[height=118mm, width=\linewidth]
{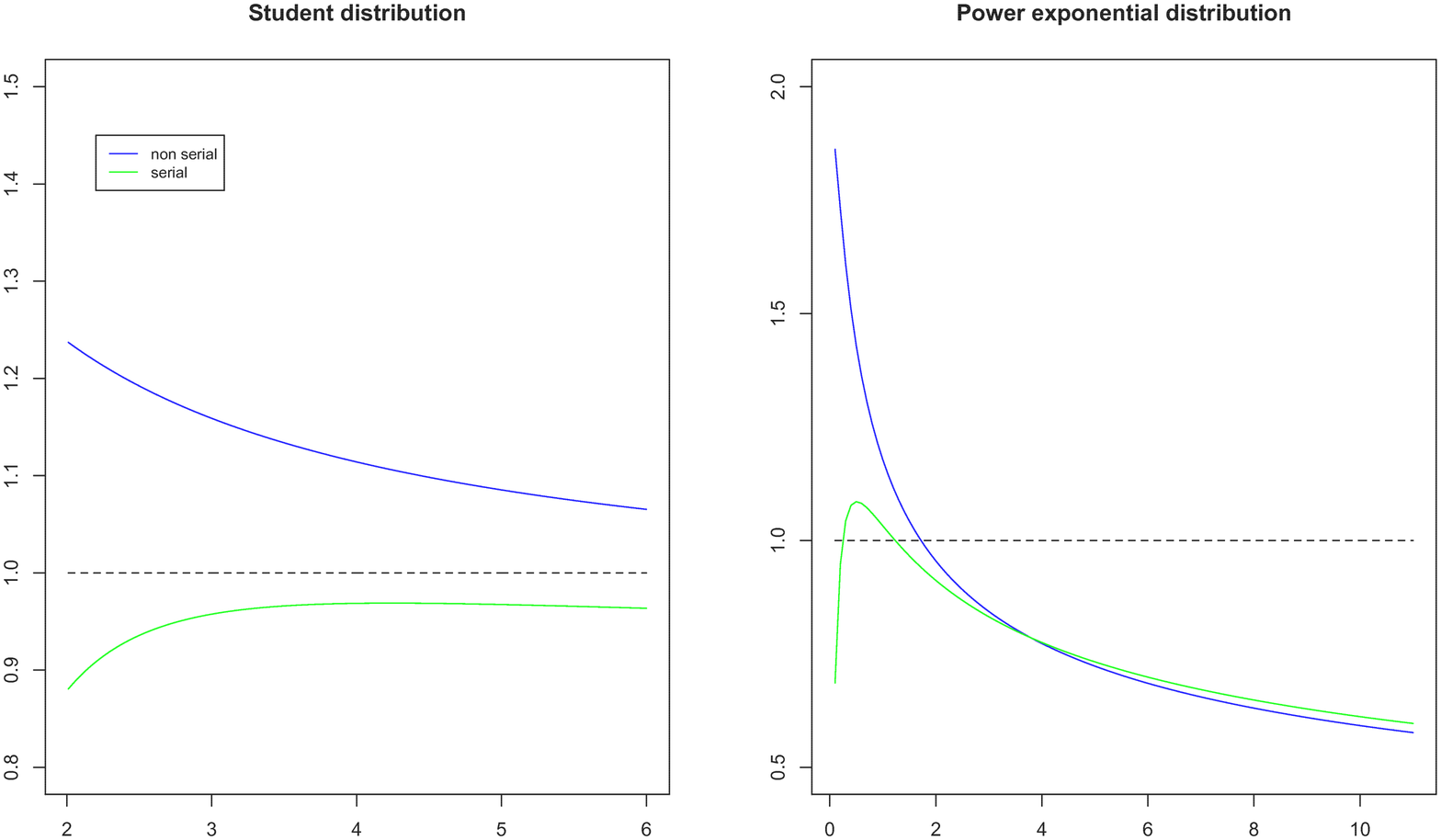}
\caption{\small Left plot : ${\rm \text{ARE}}_{f_{\nu}}(\phi_{\rm W}/\phi_{\rm vdW})$ and ${\rm
    \text{ARE}}^{\star}_{f_{\nu}}(\phi^r_{\rm SWW}/\phi^r_{\rm vdW})$ for $f_{\nu}$ the Student
  distribution,  as a function of the degrees of
  freedom $\nu \in [2, 6]$. Right plot : ${\rm \text{ARE}}_{f_{\alpha}}$ and ${\rm
    \text{ARE}}^{\star}_{f_{\alpha}}$ for  the power
  exponential   densities $f_{\alpha}$ \eqref{eq:powexp}, as a function of the
  shape parameter $\alpha \in [0, 11]$.
}
\label{SPEXplot}
\end{center}
\end{figure}

All tables seem to confirm the same findings : both the serial and
the non-serial AREs are monotone in the size of the tails, with the
non-serial ARE$_f$ attaining its maximal value ($6/\pi \approx
1.90986$) under heavy-tailed $f$ 
densities, while the maximal value for the serial ARE$_f^*$ lies
somewhere around $(6/\pi)(3/\pi )\approx 1.82346$. Inspection of Table
\ref{hodges} reveals that, although the limit of $C_f$ as $a\to 0$ is monotone in the
parameter $\epsilon$, the  ratio $D_f$ is not; from Table~3, the highest values
of $D_f$  under the distribution \eqref{eq:2d}  are attained for $a \to
\infty$ and $\epsilon\approx 0$. 

Under Student densities $f=f_{t_\nu}$, the nonserial ARE$_f$ is decreasing with $\nu$, taking value 1.41277 at the Cauchy ($\nu =1$), value one about $\nu = 15.42$ (a value of~$\nu$ that is not shown in the figure; Wilcoxon is thus outperforming van der Waerden up to $\nu =15$ degrees of freedom, with van der Waerden taking over from~$\nu =16$ on), and tending to the Gaussian value $0.95493$ as $\nu\to\infty$;  the serial ARE$_f^*$ is undefined for~$\nu\leq 2$, increasing for small values of $\nu$, from an infimum of 0.878736 (obtained as $\nu\downarrow 2$) up to a maximum of  0.968852 (reached about $\nu =Ê 4.24$), then slowly decreasing to the Gaussian value 0.911891 as $\nu\to\infty$. Sperman-Wald-Wolfowitz and Kendall  thus never outperform van der Waerden autocorrelations under Student densities.

Under the double exponential densities $f=f_\alpha$,   the nonserial ARE$_f$ is decreasing with $\alpha$, with a supremum of $6/\pi$ (the Hodges-Lehmann bound, obtained as $\alpha\downarrow 0$), and reaches value one about $\alpha = 1.7206$  (similar local asymptotic performances of Wilcoxon and van der Waerden, thus, occur at power-exponentials with parameter $\alpha = 1.7206$);   the serial ARE$_f^*$ is quite bad as $\alpha\downarrow 0$, then rapidly  increasing for small values of $\alpha$, with a maximum of 1.08552 about $\alpha = 0.510$, then deteriorating again as $\alpha\to\infty$; for $\alpha $ larger than 3, the serial and nonserial AREs roughly coincide.

\section*{Acknowledgments}
 This note originates in  a research visit  by the last two
 authors to the Department of Operations Research and Financial
 Engineering (ORFE) at  Princeton University in the Fall of 2012;
 ORFE 's support and hospitality is gratefully acknowledged.   
  Marc Hallin's research is supported by the Sonderforschungsbereich
 ``Statistical modelling of nonlinear dynamic processes" (SFB~823) of
 the Deutsche Forschungsgemeinschaft, a Discovery Grant of the Australian Research Council, and the IAP research network grant~P7/06 of the Belgian government (Belgian Science
Policy).  We gratefully acknowledge the pertinent comments by an anonymous  referee on the original version of the manuscript, which lead to substantial  improvements.

\bibliographystyle{plain}

\end{document}